\documentclass[11pt,a4paper,reqno]{article}

\usepackage{amsmath}
\usepackage{amsthm}
\usepackage{amssymb}
\usepackage{graphicx}
\usepackage{enumerate}

\newcommand{\xZ}{{\mathbb Z}}
\newcommand{\xQ}{{\mathbb Q}}

\newcommand{\xN}{{\mathbb N}}
                                   
\newcommand{\calA}{{\mathcal A}}
\newcommand{\calL}{{\mathcal L}}
\newcommand{\calC}{{\mathcal C}}

\newcommand{\bA}{\begin{gathered}\boxed{\:\begin{gathered}0 \\ 0\end{gathered}\:}\\A\end{gathered}}
\newcommand{\bC}{\begin{gathered}\boxed{\:\begin{gathered}1 \\ 1\end{gathered}\:}\\C\end{gathered}}
\newcommand{\bAC}{\begin{gathered}\boxed{\:\begin{gathered}0\ 1 \\ 1\ 0\end{gathered}\:}\\B\end{gathered}}

\newcommand{\coloneq}{\mathrel{\mathop:}=}
\newcommand{\mat}[1]{\boldsymbol{#1}}

\usepackage{color}

\newtheorem{lmm}{Lemma}
\newtheorem{thrm}[lmm]{Theorem}

\theoremstyle{definition}

\newtheorem{dfntn}[lmm]{\emph{Definition}}

\theoremstyle{remark}

\newtheorem{rmrk}[lmm]{\emph{Remark}}

\begin{document}


\title{Morphisms fixing words associated \\ with exchange of three intervals}

\author{P. Ambro\v{z}, Z. Mas\'akov\'a, E. Pelantov\'a\\
Doppler Institute \& Department of Mathematics\\
FNSPE, Czech Technical University in Prague\\
Trojanova 13, 120 00 Praha 2, Czech Republic}

\maketitle

\begin{abstract}
We consider words coding exchange of three intervals with
permutation (3,2,1), here called 3iet words. Recently, a
characterization of substitution invariant 3iet words was
provided. We study the opposite question: what are the morphisms
fixing a 3iet word? We reveal a narrow connection of such
morphisms and morphisms fixing Sturmian words using the new notion
of amicability.
\end{abstract}

\section{Introduction}

Words coding exchange of three intervals represent one of possible
generalizations of Sturmian words to a ternary alphabet. An
exchange of three intervals is given by a permutation $\pi$ on the
set $\{1,2,3\}$, and a triplet of positive numbers $\alpha$,
$\beta$, $\gamma$, corresponding to lengths of intervals $I_A$,
$I_B$, $I_C$, respectively, which define a division of the
interval $I$. In this paper we study infinite words coding
exchange of three intervals with the permutation $(3,2,1)$. Such
words are called here 3iet words. Properties of 3iet words have
been studied from various points of view in
papers~\cite{adamczewski-jtnb-14,boshernitzan-carroll-jam-72,ferenczi-holton-zamboni-aif-51,ferenczi-holton-zamboni-jam-89,ferenczi-holton-zamboni-jam-93}.

Recently, articles~\cite{balazi-masakova-pelantova-subst_inv}
and~\cite{abmp} gave a characterization of 3iet words invariant
under a substitution. Recall that a similar question for Sturmian
words (i.e.\ words coding exchange of two intervals) has been
partially solved in~\cite{crisp-jtnb-5,komatsu-jjm-22,parvaix-jtnb-11}.
Complete solution to the task was provided by
Yasutomi~\cite{yasutomi-dm-2}. An alternative proof valid for
bidirectional Sturmian words is given
in~\cite{balazi-masakova-pelantova-i-5}, yet another proof
in~\cite{beir-tia}.

One has also asked the question from another angle: what are the
substitutions fixing a Sturmian word, this problem has been
studied in a wider context. One considers the so-called Sturmian
morphisms, i.e.\ morphisms that preserve the set of Sturmian
words. The monoid of Sturmian morphisms has been described
in~\cite{seebold-tcs-88,mignosi-seebold-jtnb-5}. It turns out that
it is generated by three simple morphisms, namely
\begin{equation}\label{e:ff}
\varphi: \
\begin{aligned}
0 & \mapsto 01\\
1 & \mapsto  0
\end{aligned}\,,\qquad
\psi: \
\begin{aligned}
0 & \mapsto 10\\
1 & \mapsto  0
\end{aligned}\,,\qquad\text{and}\qquad
E: \
\begin{aligned}
0 & \mapsto  1\\
1 & \mapsto  0
\end{aligned}\,.
\end{equation}
It is known~\cite{berstel-seebold-bbms-1} that a morphism $\xi$ such that $\xi(u)$ is
Sturmian for at least one Sturmian word $u$ belongs also to the
monoid. In particular, all morphisms fixing Sturmian words are
Sturmian morphisms.

The aim of this paper is to describe morphisms over the alphabet
$\{A,B,C\}$ fixing a 3iet word. The main tool which we use is a
narrow connection between 3iet words and Sturmian words over the
alphabet $\{0,1\}$ by means of morphisms
$\sigma_{01},\sigma_{10}:\{A,B,C\}^*\rightarrow\{0,1\}^*$ given by
\begin{equation}
\label{eq:002}
\sigma_{01}:\
\begin{aligned}
  A & \mapsto 0 \\
  B & \mapsto 01 \\
  C & \mapsto 1
\end{aligned}\,,
\qquad\text{and}\qquad \sigma_{10}:\
\begin{aligned}
  A & \mapsto 1 \\
  B & \mapsto 10 \\
  C & \mapsto 0
\end{aligned}\,.
\end{equation}

In~\cite{abmp} the following statement is proved.

\begin{thrm}[\cite{abmp}]\label{thm:1}
A ternary word $u$ is a 3iet word if and only if both
$\sigma_{01}(u)$ and $\sigma_{10}(u)$ are Sturmian words.
\end{thrm}

Another important statement connecting 3iet words and Sturmian
words is taken from~\cite{balazi-masakova-pelantova-subst_inv}.

\begin{thrm}[\cite{balazi-masakova-pelantova-subst_inv}]\label{thm:2}
A non-degenerate 3iet word $u$ is invariant under a substitution
if and only if both $\sigma_{01}(u)$ and $\sigma_{10}(u)$ are
invariant under substitution.
\end{thrm}

The paper is organized as follows. In Section~\ref{sec:preli} we
recall the definitions of 3iet words and morphisms and the
geometric representation of a fixed point of a morphism. In
Section~\ref{sec:amicable} we define a relation on the set of
Sturmian morphisms with a given incidence matrix, called
amicability, and we show how to construct from a pair of amicable
morphisms a morphism over the alphabet $\{A,B,C\}$ with a 3iet
fixed point (Theorem~\ref{thm:xy}). In
Section~\ref{sec:ternarizace} we show, that any morphism $\eta$
fixing a non-degenerate 3iet word (or its square $\eta^2$) is
constructed in this way (Theorem~\ref{thm:001}).

\section{Preliminaries}\label{sec:preli}

\subsection{Three interval exchange}

A transformation $T:I\to I$ of an exchange of three intervals is
usually defined as a mapping with the domain
$I=[0,\alpha+\beta+\gamma)$, where $\alpha,\beta,\gamma$ are
arbitrary positive numbers determining the splitting of $I$ into
three disjoint subintervals $I=I_A\cup I_B\cup I_C$. An infinite
word associated to such a transformation is given as a coding of
an initial point $x_0\in I$ in a ternary alphabet $\{A,B,C\}$.
Properties of the transformation $T$ and the corresponding
infinite word do not depend on absolute values of
$\alpha,\beta,\gamma$, but rather on their relative sizes. As
well, translation of the interval $I$ on the real line does not
influence the corresponding dynamical system. For the study of
substitution properties of 3iet words it proved useful to consider
the definition of a 3iet mapping with parameters normalized by
$\alpha+2\beta+\gamma=1$ and a translation of the interval $I$
such that the initial point $x_0$ is the origin.

\begin{dfntn}
Let $\varepsilon,l,c$ be real numbers satisfying
\[
\varepsilon\in(0,1)\,,\qquad
\max\{\varepsilon,1-\varepsilon\}<l<1\,,\qquad 0\in[c,c+l)=:I\,.
\]
The mapping
\begin{equation}\label{eq:1}
T(x)=\left\{
\begin{array}{lll}
x+1-\varepsilon &\text{ for }\ x\in[c,c+l-1+\varepsilon) &=:I_A\,,\\
x+1-2\varepsilon &\text{ for }\ x\in[c+l-1+\varepsilon,c+\varepsilon) &=:I_B\,,\\
x-\varepsilon &\text{ for }\ x\in[c+\varepsilon,c+l) &=:I_C\,,
\end{array}
\right.
\end{equation}
is called exchange of three intervals with permutation $(3,2,1)$.
\end{dfntn}

Note that the parameter $\varepsilon$ represents the length of the
interval $I_A\cup I_B$, and $1-\varepsilon$ corresponds to the
length of $I_B\cup I_C$. The number $l$ is the length of the
interval $I=I_A\cup I_B\cup I_C$.

The orbit of the point $x_0=0$ under the transformation $T$
of~\eqref{eq:1} can be coded by an infinite word $(u_n)_{n\in\xZ}$
in the alphabet $\{A,B,C\}$, where
\begin{equation}\label{eq:2}
u_n= \begin{cases}
A &\hbox{if }\ T^n(0)\in I_A\,,\\
B &\hbox{if }\ T^n(0)\in I_B\,,\\
C &\hbox{if }\ T^n(0)\in I_C\,,
\end{cases}
\quad \text{ for } n\in\xZ.
\end{equation}

The infinite word $(u_n)_{n\in\xZ}$ is non-periodic exactly in the case
that the parameter $\varepsilon$ is irrational. Words coding the
orbit of $0$ under an exchange of intervals with the permutation
$(3,2,1)$ and an irrational parameter $\varepsilon$ are called 3iet
words.

\subsection{Words and morphisms}

An alphabet $\calA$ is a finite set of symbols. In this paper we
shall systematically use the alphabet $\{A,B,C\}$ for 3iet words,
and the alphabet $\{0,1\}$ for Sturmian words. A finite word in
the alphabet $\calA$ is a concatenation $v=v_1v_2\cdots v_n$, where
$v_i\in\calA$ for all $i=1,2,\dots,n$. The length of the word $v$ is
denoted by $|v|=n$. The symbol $\calA^*$ denotes the set of all
finite words over $\calA$, including the empty word $\epsilon$.
Equipped with the operation of concatenation, $\calA^*$ is a monoid.
Sequences $u_0u_1u_2\cdots\in\calA^{\xN}$, $\cdots u_{-3}u_{-2}u_{-1}\in\calA^{\xZ_{<0}}$,
$\cdots u_{-3}u_{-2}u_{-1}|u_0u_1u_2\cdots\in\calA^{\xZ}$  are called
right-sided, left-sided and bidirectional infinite word,
respectively.

If for a finite word $w$ there exist (finite or infinite) words
$v^{(1)}$ and $v^{(2)}$ such that $v=v^{(1)}wv^{(2)}$, then $w$ is
said to be a factor of the (finite or infinite) word $v$. If
$v^{(1)}$ is the empty word, then $w$ is a prefix of $v$, if
$v^{(2)}=\epsilon$, then $w$ is a suffix of $v$. The set of all
factors of an infinite word $u$ is called the language of $u$ and
denoted ${\cal L}(u)$. Factors of $u$ of length $n$ form the set
$\calL_n(u)$; obviously $\calL_n(u)=\calL(u)\cap\calA^n$. The mapping
$\calC:\xN\to\xN$ given by the prescription $n\mapsto\#\calL_n(u)$ is
called the factor complexity of the infinite word $u$.

Infinite words $u$ such that the set $\{wv\in\calL(u)\mid w\hbox{ is
not a factor of }v\}$ is finite for every $w\in\calL(u)$ are called
uniformly recurrent. Right-sided Sturmian words are defined as
right-sided infinite words with factor complexity $\calC(n)=n+1$ for
all $n\in\xN$. Bidirectional Sturmian words are uniformly recurrent
bidirectional infinite words  satisfying $\calC(n)=n+1$ for all
$n\in\xN$.

For the factor complexity $\calC$ of a 3iet word it holds that
\begin{enumerate}
\item[(i)] either $\calC(n)=n+K$ for all sufficiently large $n$,
\item[(ii)] or $\calC(n)=2n+1$ for all $n\in\xN$.
\end{enumerate}
3iet words with complexity $\calC(n)=n+K$ belong to the set of the
so-called quasisturmian words, which are images of Sturmian words
under suitable morphisms. 3iet words with complexity $\calC(n)=2n+1$
are called non-degenerate 3iet words or regular 3iet words. The
factor complexity of a 3iet word is given by (i) or (ii) according
to the parameters $\varepsilon,l$: A 3iet word is non-degenerate
if and only if $l\notin\xZ[\varepsilon]:=\xZ+\varepsilon\xZ$,
see~\cite{adamczewski-jtnb-14}.

\medskip
A mapping $\xi:\calA^*\to{\mathcal B}^*$ satisfying
$\xi(wv)=\xi(w)\varphi(v)$ for all $w,v\in\calA^*$ is called
a morphism. A morphism is uniquely determined by the images
$\xi(a)$ of all letters $a\in\calA$. The action of a morphism can
be naturally extended to infinite words by
\begin{align*}
\xi(u_0u_1u_2\cdots) &= \xi(u_0)\xi(u_1)\xi(u_2)\cdots\,,\\
\xi(\cdots u_{-3}u_{-2}u_{-1}) &= \cdots\xi(u_{-3})\xi(u_{-2})\xi(u_{-1})\,,\\
\xi(\cdots u_{-3}u_{-2}u_{-1}|u_0u_1u_2\cdots) &=
\cdots\xi(u_{-3})\xi(u_{-2})\xi(u_{-1})|\xi(u_0)\xi(u_1)\xi(u_2)\cdots\,.
\end{align*}

With every morphism $\xi$ one can associate a matrix
$\mat{M}_\xi$. The matrix has $\#\calA$ rows and $\#{\mathcal B}$
columns, and
\[
(\mat{M}_\xi)_{ab} = \hbox{ number of letters $b$ in } \xi(a)\,.
\]
A morphism $\xi:\calA^*\to\calA^*$ is called primitive if some power of
the square matrix $\mat{M}_\xi$ has all elements positive. In
other words, there exists a positive integer $k$ such that for all
$a,b\in\calA$, the letter $a$ is a factor of the $k$-th iteration
$\xi^k(b)$.

An infinite word $u$ in $\calA^{\xN}$, $\calA^{\xZ_{<0}}$, $\calA^{\xZ}$ is
said to be a fixed point of a morphism $\xi:\calA^*\to\calA^*$, if
$\xi(u)=u$. It is obvious that if $u=u_0u_1u_2\cdots$ is a fixed
point of a primitive morphism $\xi$, then $\xi(u_0)=u_0w$ for a
non-empty word $w$, and $u$ is the limit of finite words
$\xi^n(a)$, which is usually denoted by
$\xi^\infty(a)=\lim_{n\to\infty}\xi^n(a)$. Analogous properties
must be satisfied by primitive morphisms fixing left-sided or
bidirectional infinite words.

Morphisms with the above properties are sometimes called
substitutions. It is quite obvious that the only non-primitive
morphism which can fix a 3iet word or a Sturmian word is the
identity. Therefore it is not misleading not to distinguish
between notions of primitive morphism and substitution when
speaking about substitution invariant Sturmian or 3iet words.

Substitution invariance of non-degenerate bidirectional 3iet words
has been studied in~\cite{balazi-masakova-pelantova-subst_inv}.
Similarly as in the case of Sturmian words, one needs the notion
of Sturm numbers. The original definition of a Sturm number uses
continued fractions. We cite the equivalent definition given
in~\cite{allauzen-jtnb-10}: A real number $\varepsilon\in(0,1)$ is
called a Sturm number, if it is a quadratic irrational with
algebraic conjugate $\varepsilon'\notin(0,1)$.

Let us cite here the characterization of substitution invariant
3iet words from~\cite{balazi-masakova-pelantova-subst_inv}.

\begin{thrm}[\cite{balazi-masakova-pelantova-subst_inv}]\label{thm:7}
Let $u$ be a non-degenerate 3iet word with parameters
$\varepsilon,l,c$. Then $u$ is invariant under a primitive
morphism if and only if
\begin{itemize}
\item 
  $\varepsilon$ is a Sturm number
\item
  $c,l\in\xQ(\varepsilon)$
\item
  $\min\{\varepsilon',1-\varepsilon'\} \leq \, -c' \, \leq
  \max\{\varepsilon',1-\varepsilon'\}$ and 
  $\min\{\varepsilon',1-\varepsilon'\} \leq  \, c'+l' \, \leq
  \max\{\varepsilon',1-\varepsilon'\}$,
  where $x'$ is the field conjugate of $x$ in $\xQ(\varepsilon)$.
\end{itemize}
\end{thrm}

\subsection{Geometric representation of a fixed point of a morphisms}

It is useful to reformulate the task of searching for a
substitution fixing a given infinite word
$u=(u_n)_{n\in\xZ}\in\calA^{\xZ}$ in geometric terms. Let us associate
with letters of the alphabet mutually distinct lengths by an
injective mapping $\ell:\calA\to(0,+\infty)$. Then, with the infinite
word $u$ we associate a strictly increasing sequence
$(t_n)_{n\in\xZ}$ such that
\[
t_0=0\quad\text{ and }\quad t_{n+1}-t_n=\ell(u_n)\ \text{ for all
$n\in\xZ$.}
\]
A number $\Lambda>1$ satisfying
\[
\Lambda\Sigma := \{\Lambda t_n\mid n\in\xZ\} \ \subset \ \{t_n\mid
n\in\xZ\} = :\Sigma\,,
\]
is called a self-similarity factor of the sequence
$(t_n)_{n\in\xZ}$. Let us suppose that the assignment of lengths
$\ell$ and the self-similarity factor $\Lambda$ satisfy that to
every $a\in\calA$ there exists a finite set $P_a\subset(0,+\infty)$
such that
\begin{equation}\label{eq:vlastnostV}
[\Lambda t_n,\Lambda t_{n+1}]\cap \Sigma \ = \ \Lambda t_n +
P_a\qquad\text{ for all $n\in\xZ$ with $u_n=a$.}
\end{equation}
It means that the gap between $t_n$ and $t_{n+1}$ is after
stretching by $\Lambda$ filled by members of the original sequence
$(t_n)_{n\in\xZ}$ in the same way for all gaps corresponding to the
letter $a$.
An infinite word $u$ for which one can find a mapping
$\ell$ and a factor $\Lambda$ with the above described properties
is obviously invariant under a substitution $\xi$, where the
image $\xi(a)$ is determined by the distances between
consecutive elements of the set $P_a$.
We call the set $\{t_n\mid n\in\xZ\}$ with the
property~\eqref{eq:vlastnostV} the geometric representation of the
word $u$ with the factor $\Lambda$.

On the other hand, if an infinite word $u$ is invariant under a
primitive substitution $\xi$ with the matrix $\mat{M}_\xi$, then
the eigenvector of $\mat{M}_\xi$ corresponding to the dominant
eigenvalue $\Lambda$ is a column of length $\#\calA$ with all
components $x_a$, $a\in\calA$, positive, cf.~\cite{fielder-matice}.
The correspondence $\ell:a\to x_a$ results in a sequence
$(t_n)_{n\in\xZ}$ having $\Lambda$ as its self-similarity factor
and satisfying~\eqref{eq:vlastnostV}. Therefore the set $\{t_n\mid
n\in\xZ\}$ is the geometric representation of the infinite word $u$
with the factor $\Lambda$. We illustrate the concept of the
geometric representation in Figure~\ref{f}.

\begin{figure}[!ht]
\begin{center}
\includegraphics[scale=0.78]{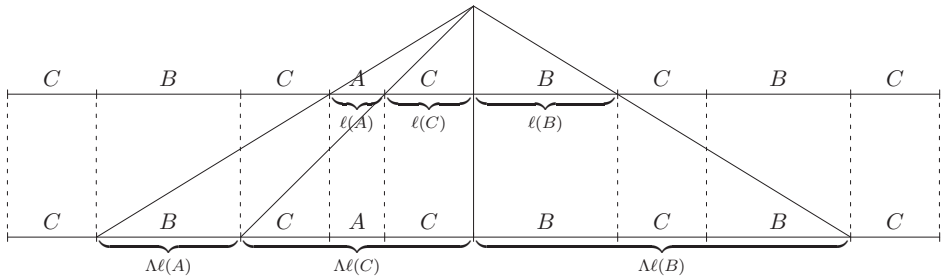}
\end{center}
\caption{Geometric representation of a ternary word $u$ fixed by a
substitution $\eta$ with the self-similarity factor $\Lambda$. In
our example, $\eta(A)=B$, $\eta(B)=BCB$, $\eta(C)=CAC$.} \label{f}
\end{figure}

In~\cite{abmp}, the authors derive (in their Corollaries 7.1 and 7.2)
several properties of matrices of substitutions fixing a 3iet
word.

\begin{thrm}[\cite{abmp}]\label{thm:8}
Let $u$ be a non-degenerate 3iet word with parameters
$\varepsilon,l,c$ which is invariant under a primitive
substitution $\eta$. Then for the dominant eigenvalue $\Lambda$ of
the matrix $\mat{M}_{\eta}$ one has
\begin{enumerate}
\item $\Lambda$ is a quadratic unit;

\item $(1-\varepsilon,1-2\varepsilon,-\varepsilon)^T$ is the
 right eigenvector of $\mat{M}_{\eta}$, corresponding to $\Lambda'$
the algebraic conjugate of $\Lambda$.
\end{enumerate}
\end{thrm}

Item (2) of the above theorem implies for the matrix $\mat{M}_\eta$
that ${\mathbf v}:=(1-\varepsilon',1-2\varepsilon',-\varepsilon')^T$ is its right
eigenvector corresponding to the dominant
eigenvalue $\Lambda$. Using Theorem~\ref{thm:7}, the parameter
$\varepsilon$ is a Sturm number, and so $\varepsilon'\notin(0,1)$.
The vector ${\mathbf v}$ has thus all components positive or all
components negative. In any case, in the geometric representation
of the fixed point of the substitution $\eta$ of Theorem~\ref{thm:8}, the length
$\ell(B)$ corresponding to the letter $B$ is the sum $\ell(A)+\ell(C)$.

\section{Amicable morphisms}\label{sec:amicable}

The narrow connection of $3$iet words and Sturmian words and their invariance
under morphisms is described in Theorems~\ref{thm:1} and~\ref{thm:2} by means of morphisms
$\sigma_{01},\sigma_{10}$, see~\eqref{eq:1}. These morphisms also allow us the description
of morphisms fixing a 3iet word using Sturmian morphisms. For that, several notions need to be defined.

\begin{dfntn}\label{d:amicablewords}
  Let $u,v$ be finite or infinite words over the alphabet $\{0,1\}$. We say that $u$ is amicable to
  $v$, and denote it by $u\propto v$, if there exist a ternary word $w$ over $\{A,B,C\}$ such that
  $u=\sigma_{01}(w)$ and $v=\sigma_{10}(w)$.
  In such a case we denote $w\coloneq\mathrm{ter}(u,v)$ and say that $w$ is the
  ternarization of $u$ and $v$.
\end{dfntn}

Note that the relation $\propto$ is not symmetric. For example,
$u=01$ is amicable to $v=10$, but not vice versa. It is also
interesting to notice that if two finite words $u,v$ satisfy
$u\propto v$, then they are of the same length and the number of
letters $a$ in $u$ and $v$ are equal for both $a=0,1$.

Figure~\ref{fff} illustrates an easy way how to recognize
amicability of two words and how to construct their ternarization.
According to the definition,  $u\propto v$ if $u$ can be written
as a concatenation $u=u^{(1)}u^{(2)}u^{(3)}\cdots$ and $v$ as a
concatenation $v=v^{(1)}v^{(2)}v^{(3)}\cdots$ such that for all
$i=1,2,3,\dots$ we have either $u^{(i)}=v^{(i)}=0$ or
$u^{(i)}=v^{(i)}=1$ or $u^{(i)}=01$ and $v^{(i)}=10$. The
ternarization $w$ is then constructed by associating letters in
the alphabet $\{A,B,C\}$ to the blocks, namely it associates $A$,
if $u^{(i)}=v^{(i)}=0$; it gives $C$ if $u^{(i)}=v^{(i)}=1$, and
it gives $B$, if $u^{(i)}=01$ and $v^{(i)}=10$.

\begin{figure}
\begin{equation}\label{e:f}
\begin{gathered}
  u = \quad\\ v =\quad \\[1mm] w =\quad
\end{gathered}
\bA\ \bC\ \bA\ \bAC\ \bA\ \bC
\end{equation}
\caption{Finite words $u=0100101$ and $v=0101001$ satisfy $u\propto v$ and their ternarization is equal to
$w=\mathrm{ter}(u,v)=ACABAC$.}
\label{fff}
\end{figure}

We introduce the notion of amicability and ternarization also for morphisms.

\begin{dfntn}\label{d:amicablemorphisms}
  Let $\varphi,\psi:\{0,1\}^*\rightarrow\{0,1\}^*$ be two morphisms.
  We say that $\varphi$ is amicable to $\psi$, and denote it by
  $\varphi\propto\psi$, if the three following
  relations hold
  \begin{equation}
  \begin{split}
    \varphi(0)&\propto\psi(0)\,,\\
    \varphi(1)&\propto\psi(1)\,,\\
    \varphi(01)&\propto\psi(10)\,.
  \end{split}
  \end{equation}
  The morphism $\eta:\{A,B,C\}^*\to\{A,B,C\}^*$ given by
\begin{align*}
\eta(A) & \coloneq\mathrm{ter}(\varphi(0),\psi(0))\,, \\
\eta(B) & \coloneq\mathrm{ter}(\varphi(01),\psi(10))\,, \\
\eta(C) & \coloneq\mathrm{ter}(\varphi(1),\psi(1))\,,
\end{align*}
is called the ternarization of $\varphi$ and $\psi$ and denoted by
$\eta\coloneq\mathrm{ter}(\varphi,\psi)$.
\end{dfntn}

As an example, consider two basic Sturmian morphisms $\varphi$,
$\psi$ from~\eqref{e:ff},
\[
\varphi: \
\begin{aligned}
0 &\mapsto 01\\
1 &\mapsto 0
\end{aligned}\,,\qquad
\psi: \
\begin{aligned}
0 &\mapsto 10\\
1 &\mapsto 0
\end{aligned}\,.
\]
It can be easily checked that $\varphi\propto\psi$ and that their
ternarization $\eta=\mathrm{ter}(\varphi,\psi)$ is of the form
\begin{equation}\label{e:fff}
\eta:\
\begin{aligned}
A & \ \mapsto\ \mathrm{ter}(01,10) =B\,,\\
B & \ \mapsto\ \mathrm{ter}(010,010)=ACA\,,\\
C & \ \mapsto\ \mathrm{ter}(0,0)=A\,.
\end{aligned}
\end{equation}

From the definition of amicability of words it follows that if
$u\propto v$ and $u'\propto v'$ then for their concatenation we
have $uu'\propto vv'$. As a simple consequence of this idea, we
have the following lemma.

\begin{lmm}\label{lem:amicable-induction}
  Let $u,v$ be two (finite or infinite) words over $\{0,1\}$ such that $u\propto v$,
  and let $\varphi,\psi:\{0,1\}^*\rightarrow\{0,1\}^*$ be two morphisms such that
  $\varphi\propto\psi$.
  Then $\varphi(u)\propto\psi(v)$. Moreover, if $w=\mathrm{ter}(u,v)$, then
  $\mathrm{ter}(\varphi(u),\psi(v))=\eta(w)$, where $\eta=\mathrm{ter}(\varphi,\psi)$.
\end{lmm}

\begin{rmrk}
Note that if $\varphi\propto\psi$ and $\eta=\mathrm{ter}(\varphi,\psi)$, then
\[
\varphi: \
\begin{aligned}
0 & \mapsto \sigma_{01}\eta (A) \\
1 & \mapsto \sigma_{01}\eta (C)
\end{aligned}
\qquad\text{and}\qquad
\psi: \
\begin{aligned}
0 & \mapsto \sigma_{10}\eta (A) \\
1 & \mapsto \sigma_{10}\eta (C)
\end{aligned}\,.
\]
\end{rmrk}

%
%

\begin{thrm}\label{thm:xy}
Let $\varphi,\psi:\{0,1\}^*\rightarrow\{0,1\}^*$ be two primitive
Sturmian morphisms having fixed points such that
$\varphi\propto\psi$. Then the morphism
$\eta:\{A,B,C\}^*\rightarrow\{A,B,C\}^*$ given by
$\eta=\mathrm{ter}(\varphi,\psi)$
has a 3iet fixed point.
\end{thrm}

\begin{proof}
The first step is to prove that a fixed point of $\varphi$, say
$u$, is amicable to a fixed point of $\psi$, say $v$. We prove the
statement for right-sided words only, the proof for left-sided and
bidirectional fixed points follows the same lines. We will discuss
two separate cases.

\smallskip

\noindent\emph{Case A.} Let there exists a letter $X\in\{0,1\}$ such that
$\varphi(X)$ starts with $X$ and $\psi(X)$ starts with $X$.
Primitivity of $\varphi$ and $\psi$ implies that both $\varphi(X)$ and $\psi(X)$ have at
least two letters.
Therefore
\[
u=\lim_{k\rightarrow\infty}\varphi^k(X)\qquad\text{and}\qquad
v=\lim_{k\rightarrow\infty}\psi^k(X)\,.
\]
Since $X\propto X$ we have $u\propto v$ by Lemma~\ref{lem:amicable-induction}.

\smallskip

\noindent\emph{Case B.} Let the negation of Case A hold.
\begin{enumerate}[a)]
\item
  Let $\varphi(1)$ start with $1$. Then necessarily $\psi(1)$ starts with $0$ which is in contradiction
  with $\varphi(1)\propto\psi(1)$.
\item
  Let $\varphi(1)$ start with $0$. Since $\varphi$ has a fixed point, $\varphi(0)$ must start with $0$.
  Thus $\psi(0)$ does not start with $0$, which implies that $\psi(1)$ starts with $1$ since $\psi$ also
  has a fixed point. \par
  \noindent Consider $\varphi(01)$ and $\psi(10)$. Clearly, $\varphi(01)= \varphi(0)\varphi(1)$ starts with $0$ and
  $ \psi(10)=\psi(1)\psi(0)$ starts with $1$.
  Moreover, since $\varphi(01)\propto\psi(10)$, the word $\varphi(01)$ must have the prefix $01$ and
  the word $\psi(10)$ must have the prefix $10$. Therefore $u=\lim_{k\rightarrow\infty}\varphi^k(01)$
  and $v=\lim_{k\rightarrow\infty}\psi^k(10)$. Now $01\propto 10$ and therefore by Lemma~\ref{lem:amicable-induction},
  it follows that $u\propto v$.
\end{enumerate}

We have shown in all cases that the fixed points $u,v$ of the
Sturmian morphisms $\varphi\propto\psi$ satisfy $u\propto v$.
Moreover, Lemma~\ref{lem:amicable-induction} implies that if
$w=\mathrm{ter}(u,v)$, then $w=\eta(w)$, i.e.\ $w$ is the fixed
point of the ternarization of $\varphi$ and $\psi$. But since
$\sigma_{01}(w)=u$, $\sigma_{10}(w)=v$ are fixed points of
primitive Sturmian morphisms, they are Sturmian words, and
therefore the infinite word $w$ must be a 3iet word, as follows
from Theorem~\ref{thm:1}.
\end{proof}

\section{Morphisms with 3iet fixed point}\label{sec:ternarizace}

The aim of this section is to prove the following theorem.

\begin{thrm}\label{thm:001}
Let $\eta$ be a primitive substitution fixing a non-degenerate
3iet word $u$. Then there exist Sturmian morphisms $\varphi$ and $\psi$
having fixed points, such that $\varphi\propto\psi$ and
$\eta$ or $\eta^2$ is equal to $\mathrm{ter}(\varphi,\psi)$.
\end{thrm}

The proof will combine results of papers~\cite{abmp,balazi-masakova-pelantova-subst_inv} concerning
substitution invariance of non-degenerate 3iet words and of the paper~\cite{balazi-masakova-pelantova-i-5}
which solves the same question for Sturmian words. We shall study infinite words defined by~\eqref{eq:2}
under a transformation $T$ from~\eqref{eq:1} where parameters $\varepsilon,l$ satisfy additional conditions
\begin{equation}\label{eq:3}
\varepsilon\in(0,1)\setminus\xQ \quad\text{ and }\quad l\notin\xZ[\varepsilon]=\xZ+\varepsilon\xZ\,.
\end{equation}
These conditions guarantee that the corresponding infinite 3iet
word is non-degene\-rate.

According to Theorem~\ref{thm:1}, the images of a 3iet word $u$ under morphisms $\sigma_{01}$, $\sigma_{10}$
are Sturmian words. Let us determine parameters of the Sturmian words $\sigma_{01}(u)$, $\sigma_{10}(u)$
(i.e.\ the corresponding exchanges of two intervals), provided that the parameters of $u$ are $\varepsilon,l,c$.
The procedure is illustrated in Figure~\ref{ff}.

\begin{figure}
\begin{center}
\includegraphics[scale=0.78]{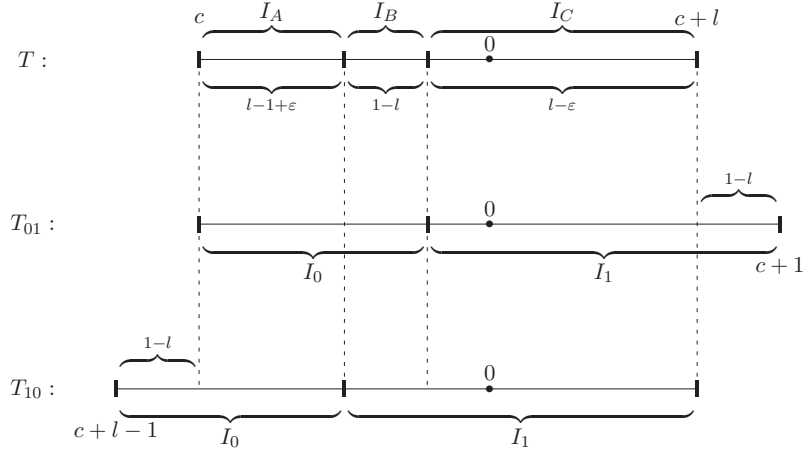}
\end{center}
\caption{Exchanges of intervals corresponding to a 3iet
word $u$ and Sturmian words $\sigma_{01}(u)$, $\sigma_{10}(u)$.}
\label{ff}
\end{figure}

Define the mapping $T_{01}:[c,c+1)\to[c,c+1)$ by
\[
T_{01}(x)=\left\{
\begin{array}{lll}
x+1-\varepsilon &\text{ for }\ x\in[c,c+\varepsilon) &=:I_0\\
x-\varepsilon &\text{ for }\ x\in[c+\varepsilon,c+1) &=:I_1
\end{array}
\right.
\]
Comparing $T_{01}$ and $T$ we obtain (see Figure~\ref{ff})
\[
x\in I_B \quad\iff\quad T_{01}(x)\in[c+l,c+1)\,.
\]
For $x\in[c,c+l)$ we have
\begin{align*}
x\in I_A &\ \implies \ x\in I_0 \ \text{ and }\ T_{01}(x)=T(x)\,,\\
x\in I_B &\ \implies \ x\in I_0,\ \ T_{01}(x)\in I_1 \ \text{ and } \ T(x)=T_{01}^2(x)\,,\\
x\in I_C &\ \implies \ x\in I_1\ \text{ and }\ T_{01}(x)=T(x)\,.
\end{align*}
Therefore $\sigma_{01}(u)$ is the infinite word coding the orbit
of 0 under the exchange $T_{01}$ of intervals with lengths
$\varepsilon$ and $1-\varepsilon$. Such a word is a Sturmian word
of the slope $\varepsilon$ and intercept $-c$ (i.e.\ the distance
of the initial point of the orbit and the left end-point of the
interval $[c,c+1)$ which is the domain of $T_{01}$).

In a similar way, we derive that the infinite word
$\sigma_{10}(u)$ is the coding of the orbit of 0 under the
exchange of two intervals $T_{10}:[c+l-1,c+l)\to[c+l-1,c+l)$. In
particular, it is a Sturmian word of the slope $\varepsilon$ and
intercept $-c-l+1$.

Let us cite the result characterizing substitution invariant
Sturmian words. Comparing~\cite{yasutomi-dm-2}
and~\cite{balazi-masakova-pelantova-i-5} we obtain that a
right-sided Sturmian word with the slope $\alpha$ and intercept
$\beta$ is substitution invariant if and only if the bidirectional
Sturmian word with the same slope and intercept is substitution
invariant.

\begin{thrm}[\cite{yasutomi-dm-2}]\label{thm:9}
Let $\alpha\in(0,1)$ be irrational and $\beta\in[0,1)$. A Sturmian
word with the slope $\alpha$ and intercept $\beta$ is invariant
under a primitive morphism if and only if
\begin{enumerate}
\item $\alpha$ is a Sturm number;
\item $\beta\in\xQ(\alpha)$;
\item $\min\{\alpha',1-\alpha'\} \leq \, \beta' \, \leq
\max\{\alpha',1-\alpha'\}$,
where $\alpha'$, $\beta'$ denote the field conjugates of
$\alpha$, $\beta$ in $\xQ(\alpha)$.
\end{enumerate}
\end{thrm}

Note that the inequalities in Item (3) are satisfied for $\beta'$
if and only they are satisfied replacing $\beta'$ by $1-\beta'$.
Knowing the slope and intercept of Sturmian words
$\sigma_{01}(u)$, $\sigma_{10}(u)$ we can deduce from Theorem~\ref{thm:7}
the statement of Theorem~\ref{thm:1}, namely that a non-degenerate 3iet
word is invariant under a primitive substitution if and only if both
Sturmian words $\sigma_{01}(u)$, $\sigma_{10}(u)$ are substitution invariant.

We will now put into relation the substitutions fixing infinite words $u$,
$\sigma_{01}(u)$, and $\sigma_{10}(u)$. First we consider the self-similarity factors
and geometric representations of these substitutions.

\begin{lmm}\label{l:9}
Let $\eta$ be a primitive substitution over the alphabet
$\{A,B,C\}$ having as its fixed point a non-degenerate 3iet word
$u$. Let us denote its parameters $\varepsilon,l,c$. Denote by
$\Lambda$ the dominant eigenvalue of the matrix $\mat{M}_\eta$ and
by $\bigl(\ell(A),\ell(B),\ell(C)\bigr)^T$ its positive right
eigenvector corresponding to $\Lambda$. If $\Lambda'>0$, then
there exist substitutions $\varphi,\psi:\{0,1\}^*\to\{0,1\}^*$
fixing $\sigma_{01}(u)$, $\sigma_{10}(u)$, respectively, and such
that $\Lambda$ is the dominant eigenvalue of $\mat{M}_\varphi$ and
$\mat{M}_\psi$, and $\bigl(\ell(A),\ell(C)\bigr)$ is their common
right eigenvector corresponding to $\Lambda$. Moreover,
$\ell(B)=\ell(A)+\ell(C)$.
\end{lmm}

\begin{proof}
Theorems~\ref{thm:7} and~\ref{thm:8} imply that ${\mathbf
v}=\bigl(\ell(A),\ell(B),\ell(C)\bigr)^T=\bigl(1-\varepsilon',1-2\varepsilon',-\varepsilon'\bigr)^T$
is a right eigenvector of $\mat{M}_\eta$ corresponding to
$\Lambda$. Recall that $\sigma_{01}(u)$ is the Sturmian word of
the slope $\varepsilon$ and intercept $-c$, and $\sigma_{10}(u)$
the Sturmian word of the slope $\varepsilon$ and intercept
$-c-l+1$. By Theorem~\ref{thm:2} they are invariant under
substitutions, say $\phi$, $\psi$. Since $\varepsilon$ and
$1-\varepsilon$ are densities of letters $0$ and $1$ respectively,
the substitution matrices $\mat{M}_\varphi$ and $\mat{M}_\psi$
must have the eigenvector
$\bigl(1-\varepsilon',-\varepsilon'\bigr)^T=\bigl(\ell(A),\ell(C)\bigr)^T$.
Obviously $\ell(B)=\ell(A)+\ell(C)$.

It remains to show that $\varphi,\psi$ can be chosen so that the
dominant eigenvalue of $\eta$, i.e.\ $\Lambda$, is also the
dominant eigenvalue of $\mat{M}_\varphi$, $\mat{M}_\psi$. As a
consequence of the equality $\mat{M}_\eta{\mathbf
v}=\Lambda{\mathbf v}$ and the fact that $\Lambda$ is a quadratic
unit, we have $\xZ+\varepsilon'\xZ=:\xZ[\varepsilon'] = \Lambda
\xZ[\varepsilon'] = \Lambda'\xZ[\varepsilon']$, which, after
conjugation, gives
\begin{equation}\label{eq:11}
\xZ[\varepsilon] = \Lambda' \xZ[\varepsilon] = \Lambda\xZ[\varepsilon]\,.
\end{equation}

Proposition~5.6 of~\cite{balazi-masakova-pelantova-subst_inv}
(see also Remarks~5.7 and 6.5 ibidem) implies  that $\Lambda'c\in
c+\xZ[\varepsilon]$ and $\Lambda'(c+l-1+\varepsilon)\in
c+l-1+\varepsilon + \xZ[\varepsilon]$. This, together
with~\eqref{eq:11} gives
\begin{equation}
\begin{split}
\Lambda'(c+\xZ[\varepsilon]) &= c+\xZ[\varepsilon]\,,\\
\Lambda'(c+l-1+\varepsilon+\xZ[\varepsilon])&= c+l-1+\varepsilon
+\xZ[\varepsilon]\,.
\end{split}
\end{equation}
Note that the assumption $\Lambda'>0$ is required in order that we
can use results from~\cite{balazi-masakova-pelantova-subst_inv}.

Realize that substitution invariance of $\sigma_{01}(u)$ and
$\sigma_{10}(u)$ implies by Theorem~\ref{thm:9} that their
parameters satisfy
\[
\min\{\varepsilon',1-\varepsilon'\} \leq \, -c' \, \leq
\max\{\varepsilon',1-\varepsilon'\},\quad
\min\{\varepsilon',1-\varepsilon'\} \leq  \, c'+l' \, \leq
\max\{\varepsilon',1-\varepsilon'\}\,.
\]
These inequalities, together with~\eqref{eq:11}, already imply
that there exist substitutions $\varphi$ and $\psi$ with factor
$\Lambda$ (see proof of Proposition~5.3
in~\cite{balazi-masakova-pelantova-i-5}).
\end{proof}

\begin{proof}[Proof of Theorem~\ref{thm:001}]
The dominant eigenvalue of the matrix $\mat{M}_\eta$ is a
quadratic unit $\Lambda$. If $\Lambda'>0$, we shall prove the
statement for $\eta$. If $\Lambda'<0$, we will consider the second
iteration $\eta^2$. Therefore we consider without loss of
generality $\Lambda'>0$.

With the help of geometric representation of infinite words we
will show that morphisms $\varphi,\psi$ found by Lemma~\ref{l:9}
are amicable, i.e.\ $\varphi\propto\psi$, and that $\eta$ is their
ternarization. We use the fact that all of the considered
substitutions, $\eta,\varphi$ and $\psi$ have the same factor
$\Lambda$. The idea of the proof is illustrated in
Figure~\ref{ffff}.

\begin{figure}[!ht]
\centering
\begin{flushleft}
Ternary substitution $\eta:\ A\mapsto B$, $B\mapsto BCB$, $C\mapsto CAC$ and its fixed point $u$
\end{flushleft}
\includegraphics[scale=0.78]{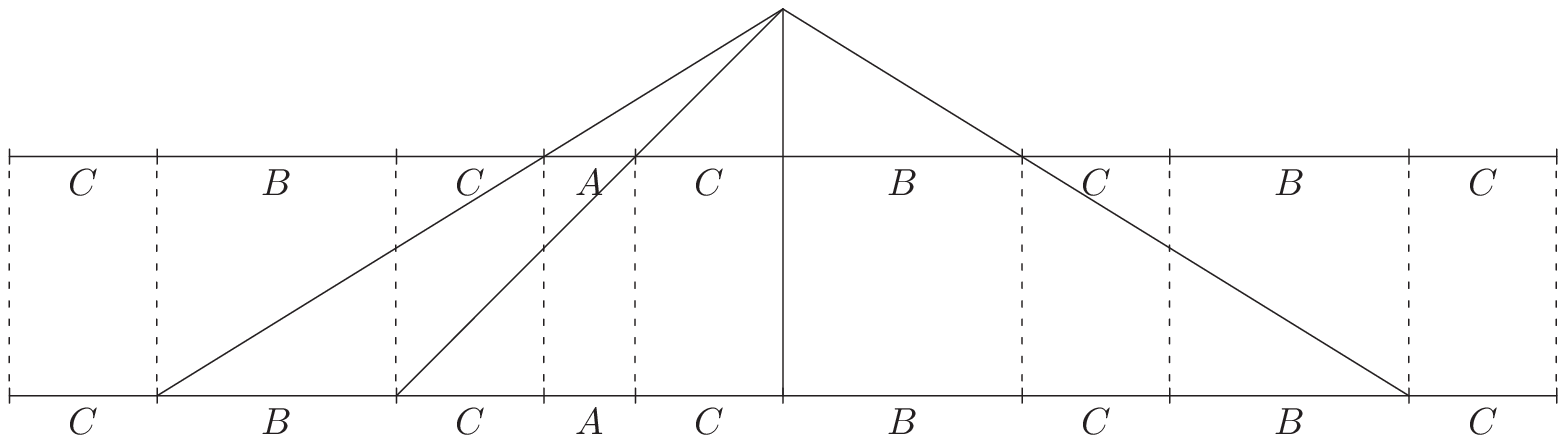}
\vskip0.5cm
\begin{flushleft}
Sturmian substitution $\varphi = \sigma_{01}\circ\eta:\ 0\mapsto 01,\ 1\mapsto 101$ and its fixed point $\sigma_{01}(u)$
\end{flushleft}
\includegraphics[scale=0.78]{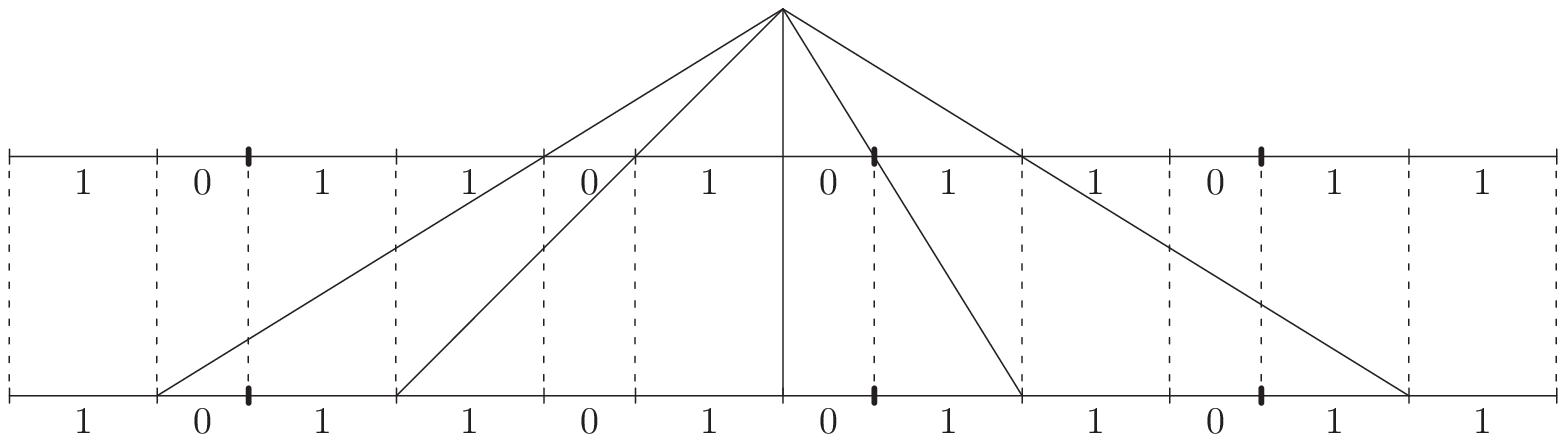}
\vskip0.5cm
\begin{flushleft}
Sturmian substitution $\psi=\sigma_{10}\circ\eta:\ 0\mapsto 10,\ 1\mapsto 101$ and its fixed point $\sigma_{10}(u)$
\end{flushleft}
\includegraphics[scale=0.78]{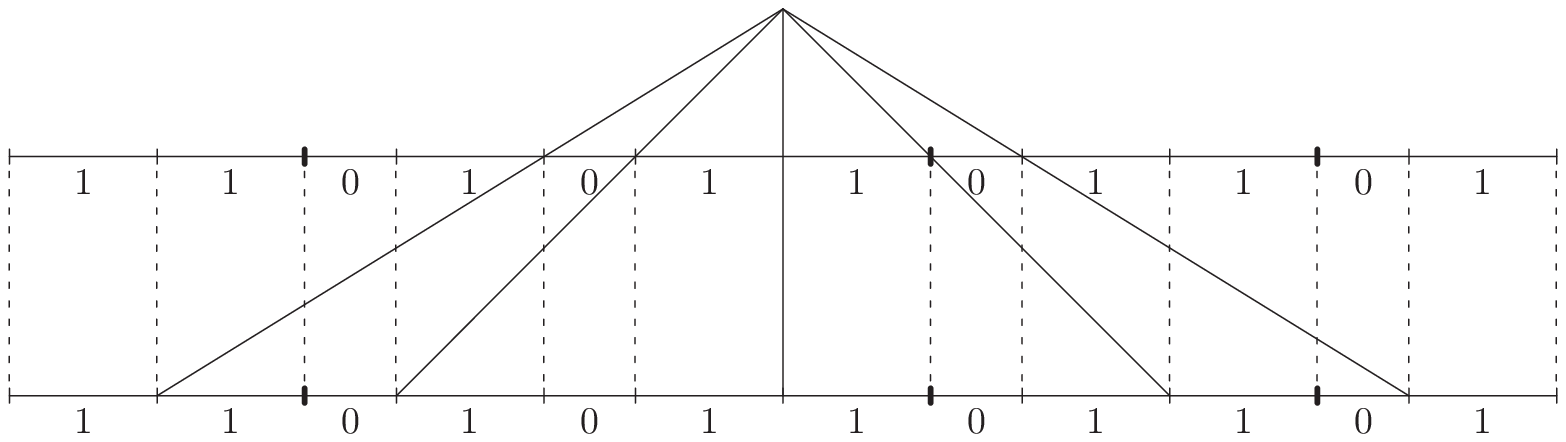}
\caption{Geometric representation of infinite words $u$, $\sigma_{01}(u)$
and $\sigma_{10}(u)$, and the substitutions $\eta$, $\varphi$, $\psi$ (all with the same self-similarity 
factor $\Lambda$) fixing them. We have $u=\mathrm{ter}(\sigma_{01}(u),\sigma_{10}(u))$ and 
$\eta=\mathrm{ter}(\varphi,\psi)$.}
\label{ffff}
\end{figure}

Let $u$ be a fixed point of $\eta$ and let $\{t_n\}_{n=0}^\infty$
be the geometric representation of the substitution $\eta$ with
the dominant eigenvalue $\Lambda$ and the right eigenvector
$\bigl(\ell(A),\ell(B),\ell(C)\bigr)^T$ for which
\[
t_{n+1}-t_n = \ell(X)\quad\iff\quad u_n=X\in\{A,B,C\}\,.
\]
Morphisms $\varphi$ and $\psi$ are Sturmian substitutions with
fixed points $\sigma_{01}(u)$, $\sigma_{10}(u)$, respectively. The
geometric representation of the infinite word $\sigma_{01}(u)$ is
\[
\{t_n^{01}\}_{n=0}^\infty:= \{t_n\}_{n=0}^\infty \cup
\{t_n+\ell(A)\mid u_n=B\}\,,
\]
and the geometric representation of the infinite word
$\sigma_{10}(u)$ is
\[
\{t_n^{10}\}_{n=0}^\infty:= \{t_n\}_{n=0}^\infty \cup
\{t_n+\ell(C)\mid u_n=B\}\,.
\]
If $t_{n+1}-t_n=\ell(A)$, i.e.\ $u_n=A$, then the segment in
$\{t_n\}_{n=0}^\infty$ between $\Lambda t_n$ and $\Lambda t_{n+1}$
(both in $\{t_n\}_{n=0}^\infty$) contains points ordered according
to $\eta(A)$. And the segment in $\{t_n^{01}\}_{n=0}^\infty$
between $\Lambda t_n$ and $\Lambda t_{n+1}$ (both in $\{t_n^{01}\}_{n=0}^\infty$) contains points
ordered according to $\sigma_{01}\bigl(\eta(A)\bigr)$.
Analogically, for $n$ such that $u_n=C$, points in $\{t_n^{01}\}_{n=0}^\infty$
between $\Lambda t_n$ and $\Lambda t_{n+1}$ are ordered according to
$\sigma_{01}\bigl(\eta(C)\bigr)$.

From what was said above it is obvious, that the substitution $\varphi$ with factor $\Lambda$ fixing
the Sturmian word $\sigma_{01}(u)$ must be of the form
\[
\varphi: \
\begin{aligned}
0 & \mapsto \sigma_{01}\eta (A)\\
1 & \mapsto \sigma_{01}\eta (C)
\end{aligned}\,.
\]
In a similar way, we can deduce that the substitution $\psi$ under which the infinite word
$\sigma_{10}(u)$ is invariant is of the form
\[
\psi: \
\begin{aligned}
0 & \mapsto \sigma_{10}\eta (A)\\
1 & \mapsto \sigma_{10}\eta (C)
\end{aligned}\,.
\]
By Definition~\ref{d:amicablewords}, we have that
$\varphi(0)\propto\psi(0)$ and $\varphi(1)\propto\psi(1)$,
and that $\eta(A)=\mathrm{ter}(\varphi(0),\psi(0))$,
$\eta(C)=\mathrm{ter}(\varphi(1),\psi(1))$.

In order to complete the proof of the theorem, we have to show that
$\varphi(01)\propto\psi(10)$ and $\eta(B)=\mathrm{ter}(\varphi(01),\psi(10))$.
For that, consider $n\in\xZ$ such that
$t_{n+1}-t_n=\ell(B)=\ell(A)+\ell(C)$, i.e.\ $u_n=B$. The
segment between $\Lambda t_n$ and $\Lambda
\bigl(t_{n}+\ell(A)\bigr)$ in the geometric representation
$\{t_n^{01}\}_{n=0}^\infty$ of $\sigma_{01}(u)$ contains the
points arranged according to $\sigma_{01}\eta(A)$. Similarly, the
segment between $\Lambda \bigl(t_{n}+\ell(A)\bigr)$ and $\Lambda
t_{n+1}$ contains the points arranged according to
$\sigma_{01}\eta(C)$. Of course, the segment between $\Lambda t_n$
and $\Lambda t_{n+1}$ in the geometric representation
$\{t_n\}_{n=0}^\infty$ of the original infinite word $u$ is
arranged according to $\eta(B)$. Altogether, we have
\[
\sigma_{01}\eta(B) = \sigma_{01}\eta(A)\sigma_{01}\eta(C) =
\varphi(0)\varphi(1)\,.
\]
Analogously,
\[
\sigma_{10}\eta(B) = \sigma_{10}\eta(C)\sigma_{10}\eta(A) =
\psi(1)\psi(0)\,.
\]
This means that $\varphi(01)\propto\psi(10)$, and the word
$\eta(B)$ is the ternarization of words $\varphi(01)$ and
$\psi(10)$. Consequently, $\varphi$ is amicable to $\psi$, and the
substitution $\eta$ is the ternarization of $\varphi$ and $\psi$.
\end{proof}

\section*{Acknowledgements}

We acknowledge financial support by the grants MSM6840770039 and
LC06002 of the Ministry of Education, Youth, and Sports of the
Czech Republic.



\end{document}